\documentclass[10pt]{amsart}
\usepackage{amssymb, amscd, amsmath, amsthm, 
epsf, epsfig, latexsym, enumerate}

\renewcommand{\leq}{\leqslant}

\newtheorem{theorem}{Theorem}
\newtheorem{lemma}[theorem]{Lemma}

\begin{document}

\title{knot modules and ribbon 2-knots}

\author{Jonathan A. Hillman }
\address{School of Mathematics and Statistics\\
     University of Sydney, NSW 2006\\
      Australia }

\email{jonh@maths.usyd.edu.au}

\begin{abstract}
We show that every $\mathbb{Z}$-torsion free knot module 
is realized by a ribbon 2-knot with group $\pi$ of geometric dimension $\leq2$,
and give some partial results on the characterization of the knot modules of fibred ribbon 2-knots.
\end{abstract}



\maketitle

\section{Introduction}

The question underlying this note is whether knot groups 
with deficiency 1 are the groups of ribbon 2-knots.
Every group with deficiency 1 and weight 1 is the group of a 2-knot.
It is the group of a ribbon 2-knot if and only if 
it has a Wirtinger presentation of deficiency 1 \cite{Ya69}
(see also \cite [Section 1.7]{Hi1});
equivalently, if there is a deficiency-1 presentation such that the 
balanced presentation of the trivial group obtained by killing a meridian
is AC equivalent to the empty presentation \cite{Yo82}.
There is a parallel criterion in the fibred case.
A knot group is the group of a fibred homotopy ribbon 2-knot if and only if 
it has deficiency 1 and its commutator subgroup is finitely generated \cite[Theorem 17.9]{Hi2}.
If the Andrews-Curtis Conjecture holds then every group $\pi$ 
with  deficiency 1 and weight 1 is a ribbon 2-knot group,
while if the Whitehead Asphericity Conjecture holds then 
every such group has geometric dimension $\leq2$.

We shall show that every $\mathbb{Z}$-torsion free knot module 
is realized by a ribbon 2-knot with group $\pi$ such that $g.d.\pi\leq2$.
After first posting this paper on the arXiv, we learned that 
(apart from the consideration of geometric dimension)
Theorem 1 below is a special case of Theorem 3.2 of \cite{Ka08}, 
which characterizes the homology of infinite cyclic covers of ribbon links of surfaces in $S^4$.
However, our motivation for considering this issue is different,
and we hope that the present simple argument may still be of interest.

Whether knot modules which are finitely generated as abelian groups
can be realized by {\it fibred} ribbon 2-knots seems more difficult, 
and we have only partial results on this. 
Knot modules which are direct sums of cyclic modules 
and which are torsion-free and finitely generated as groups 
are realized by fibred ribbon knots.
(This is surely known, but we include a proof as Theorem 2.)
Beyond this, every knot module which is torsion-free and finitely generated as a group
may be realized either by a fibred 2-knot in a homology 4-sphere (Lemma 3)
or by a ribbon 2-knot whose group is an ascending HNN extension
(Lemma 4).

The final section records seven questions on the issues considered here;
deficiency, geometric dimension, knot modules and (fibred) ribbon knots.

\newpage
\section{terminology}
{\it 
A Wirtinger presentation\/} for a group $G$ is one in which each relator 
expresses one of the generators as a conjugate of itself or another.
Thus it has the form
\[
\langle{x_1,\dots,x_m}\mid{x_j=w_{ij}x_iw_{ij}^{-1}~\forall(i,j)\in{S}}\rangle,
\]
where each $w_{ij}$ is a word in the generators, indexed by a subset $S$ of $\{1,\dots,m\}^2$.
After introducing equal numbers of new generators and relators, we may assume that the
words $w_{ij}$ have length 1. We may then arrange that they have positive exponent.
The resulting presentation is determined by a {\it labeled oriented graph} $(\Gamma,\lambda)$
with vertices $V$, edges $E\subseteq{V^2}$, origin and terminus functions
$o,t:E\to{V}$, and a third (label) function $\lambda:E\to V$, 
and may be given by 
\[
\langle{V}\mid{t(e)=\lambda(e)o(e)\lambda(e)^{-1)}~\forall{e}\in{E}}\rangle.
\]
If the group $G$ has abelianization $G/G'\cong\mathbb{Z}$ and the presentation has deficiency 1
then $\Gamma$ is a tree, and so we have a {\it labelled oriented tree}, or LOT \cite{Ho85}.

\section{realizing knot modules by ribbon 2-knots}

Let $\Lambda=\mathbb{Z}[t,t^{-1}]$, and let  $\varepsilon:\Lambda\to\mathbb{Z}$
be the canonical augmentation homomorphism, determined by $\varepsilon(t)=1$.
A {\it knot module\/} is a finitely generated $\Lambda$-module $A$ such that
$\mathbb{Z}\otimes_\Lambda{A}=0$.
We shall say that $A$ can be realized by an $n$-knot $K$ 
if $A\cong\pi'/\pi''$, where $\pi=\pi{K}$ is the knot group.
If $\pi$ has deficiency 1 then $\pi'/\pi''$  is $\mathbb{Z}$-torsion free.
A $\mathbb{Z}$-torsion free knot module has finite rank as an abelian group 
and a square presentation matrix as a $\Lambda$-module \cite{Le77}.

\begin{theorem}
Every $\mathbb{Z}$-torsion free knot module is realized by a ribbon $2$-knot 
with group $\pi$ such that $g.d.\pi\leq2$.
\end{theorem}

\begin{proof}
Let $A$ be a $\mathbb{Z}$-torsion free knot module.
Then $A$ has a presentation matrix of the form $\alpha=tM+(I_r-M)$,
where $M$ is an integer $r\times{r}$ matrix such that $\det(M)\not=0$ and
$\det(M-I_r)\not=0$ \cite{Tr74}.
Let $\pi$ be the group with presentation
\[
\langle{t,~x_1,\dots,x_r, y_1,\dots,y_r}
\mid{y_i=\Pi_jx_i^{m_{ij}}},
~ty_it^{-1}y_i^{-1}=x_i^{-1},
~\forall1\leq{i}\leq{r}
\rangle.
\]
(The order of the terms in the products $y_i=\Pi_jx_i^{m_{ij}}$
is not important.)
Then $\pi'$ is the normal closure of $\{x_1,\dots,x_r\}$.
Let $Y$ and $Z$ be the subgroups of $F(r)=F(x_1,\dots,x_r)$ 
generated by $\{y_1,\dots,y_r\}$ and $\{x_1^{-1}y_1,\dots,x_r^{-1}y_r\} $, respectively.
Since each has image of rank $r$ in $F(r)/F(r)'=\mathbb{Z}^r$,
and since  finitely generated free groups are Hopfian,
they are freely generated by these bases.
Hence $\pi$ is an HNN extension with base $F(r)$ and associated subgroups
$Y$ and $Z$, and so there is a $K(\pi,1)$ which is a finite 2-complex.
Moreover, it easy to see that $\pi'/\pi''\cong{A}$.

Now let $s_i=y_ity_i^{-1}$, for $1\leq{i}\leq{r}$.
Then $x_i=s_it^{-1}$, and so we can write each $y_j$ in terms of the $s_i$s and $t$.
Hence $\pi$ also has the deficiency-1 Wirtinger presentation
\[
\langle{t,~s_1,\dots,s_r}
\mid{s_i=y_ity_i^{-1}
~\forall1\leq{i}\leq{r}}
\rangle,
\]
and so $\pi$ is the group of a ribbon 2-knot.
\end{proof}

A similar argument may be used to show that the groups constructed 
in \cite[Theorem 11.1]{Le77} have Wirtinger presentations of deficiency 1.

If the knot module is cyclic then the construction of Theorem 1 gives a 1-relator group.
If $\pi$ is any 1-relator knot group then the relator is not a proper power,
and so $g.d.\pi\leq2$ \cite{Ly50}.
An argument related to that of the theorem shows that any such group 
has a deficiency-1 Wirtinger presentation.
However it need not have a 1-relator Wirtinger presentation.
Yoshikawa has shown that if a knot group has a 1-relator Wirtinger presentation 
then it is an HNN extension with free base and associated subgroups. 
He showed also that if $p,q,r,s\not =\pm1$ and $ps-qr=\pm1$ then the groups with presentation
$\langle{a,b,t}\mid{a^p=b^q},~ta^rt^{-1}=b^s\rangle$ have no such free base \cite{Yo88}.
If $r=p+1$ these groups have 1-relator presentations,
but do not have 1-relator Wirtinger presentations.
(For example, taking $p=2$, $q=r=3$ and $s=4$ gives the presentation
$\langle{b,t}\mid{(b^{-3}t^{-1}b^4t)^2=b^3}\rangle$.
Let  $s_i=b^{-i}tb^i$, for $i\in\mathbb{Z}$. 
Then we may write the relation as $b=s_5^{-1}s_1s_4^{-1}t$.
Using this relation to eliminate the generator $b$,
we obtain a deficency-1 Wirtinger presentation 
with 4 generators.)

If $g.d.\pi=2$ then $c.d.\pi=2$ and $\mathrm{def}(\pi)=1$.
If $\pi$ has deficiency 1 then $g.d.\pi=c.d.\pi$ \cite[Theorem 2.8]{Hi2}.
If $\mathcal{P}$ is a deficiency-1 presentation of a group $\pi$ of weight 1 
then adjoining a 2-cell to the  associated 2-complex $C(\mathcal{P})$ 
gives a contractible 2-complex.
Hence if the (finite) Whitehead Conjecture is true then
$C(\mathcal{P})$ is aspherical and so $g.d.\pi\leq2$.
If the Eilenberg-Ganea conjecture holds, then $c.d.\pi\leq2$ 
implies $g.d.\pi\leq2$.
It is known that these conjectures cannot both hold in general.
(However, the known counterexamples are not finitely presentable.)

If $c.d.\pi=2$ and $\mathcal{P}$ is any finite presentation for $\pi$ 
then $\pi_2(C(\mathcal{P}))$ is a finitely generated projective $\mathbb{Z}[\pi]$-module.
If, moreover, $\pi$ is of type $FF$ then $\pi_2(C(\mathcal{P}))$ is
stably free, by a Schanuel's Lemma argument.
After adjoining  2-cells along constant maps we may assume that
$\pi_2(C(\mathcal{P}))$ is a free module.
Adjoining 3-cells along a basis for this module then gives 
a finite 3-dimensional $K(\pi,1)$-complex.

\section{finitely generated modules and fibred 2-knots}

It is not known whether  every $\mathbb{Z}$-torsion free knot module $A$ which 
is finitely generated as an abelian group is realized 
by a high dimensional knot group with finitely generated commutator subgroup.
In this section we shall consider the sharper question, 
whether such a knot module is realized by  a fibred ribbon 2-knot.

For knot groups $\pi$ with $\pi'$ finitely generated the conditions
``$\pi'$ {\it is a free group}", ``$g.d.\pi\leq2$" and ``$\mathrm{def}(\pi)= 1$"  
are equivalent, while for any knot group , ``$g.d.\pi\leq2$"
if and only if  ``$c.d.\pi\leq2$ {\it and} $\mathrm{def}(\pi)=1$".
(See \cite[Chapter 2]{Hi2} for the less obvious implications.)
Knot groups with $\pi'$ free are the groups of fibred knots with fibre 
a once-punctured connected sum of copies of $S^2\times{S^1}$.
A fibred 2-knot with group $\pi$ is homotopy ribbon if and only if $\pi'$ is free \cite{Co83}.
A 2-knot $K$ is $s$-concordant to a fibred homotopy ribbon knot if and only if
$\mathrm{def}(\pi)=1$ and $\pi'$ is finitely generated \cite[Theorem 17.9]{Hi2}.
The next result is largely known, but is included for completeness.

\begin{theorem}
Let $A$ be a knot module which is a direct sum of cyclic $\Lambda$-modules
and is free as an abelian group.
Then $A$ is realized by a fibred ribbon $2$-knot.
\end{theorem}

\begin{proof}
Suppose first that $A$ is cyclic.
We may assume that
$A\cong\Lambda/(\alpha)$,
where $\alpha=\Sigma{a_k}t^k$ is a polynomial 
of degree $r$ in $\mathbb{Z}[t]$ with nonzero constant term
and $\varepsilon(\alpha)=1$.
Then $|a_0|=|a_r|=1$,
since $A$ is finitely generated as an abelian group.
Let $\beta=(\alpha-1)/(t-1)$. 
Then $\beta=\Sigma_{i=0}^{i=r-1}b_it^i\in\mathbb{Z}[t]$,
and $|b_0|=|b_{r-1}|=1$.
Let $\pi$ be the group with presentation
\[
\langle{t,x}\mid{x=wtw^{-1}t^{-1}}\rangle,
\] 
where
$w=\Pi_{i=0}^{i=r-1}{t^i}x^{b_i}t^{-i}$.
Then $\pi$ is the normal closure of $t$, and so has weight 1,
while $\pi'/\pi''\cong\Lambda/(\alpha)$, since $\alpha=(t-1)\beta+1$.
The commutator subgroup is the normal closure of $x$.
We see easily that (since $b_{r-1}=\pm1$) we can solve for
$t^rxt^{-r}$ in terms of the finite set $\{x,txt^{-1},\dots,t^{r-1}xt^{1-r}\}$.
Similary (since $b_0=\pm1$) we may solve for $t^{-1}xt$ in terms of the same set.
It follows that $\pi'$ is finitely generated.

Let $u=xt$ and rewrite $w$ as a word $W$ in $\{t,u\}$. 
Then $\pi$ also has the 1-relator Wirtinger presentation
\[
\langle{t,u}\mid{u=WtW^{-1}}\rangle.
\] 
Such presentations are realized by ribbon knots of 1-fusion \cite{Ya69}. 
(See also \cite[Theorem 7.17]{Hi1}.)
Ribbon knots of 1-fusion with finitely generated commutator subgroup
are fibred \cite{Yo81}. 
(See also \cite[Appendix]{AS88}.)

Taking sums of knots shows that if $A$ is a nontrivial direct sum of cyclic $\Lambda$-modules
then $A$ is realized by a fibred ribbon knot.
\end{proof}

In general, it seems difficult to realize a knot module $A$ by a group of deficiency 1
which has both weight 1 and also finitely generated commutator subgroup.
Taken separately, these tasks are easy.
We do not need the weight-1 condition to find a fibred knot in 
an homology 4-sphere which realizes $A$.

\begin{lemma} 
Let $A$ be a knot module which is free of rank $r$ as an abelian group.
Then $A$ is realized by a fibred $2$-knot in an homology $4$-sphere.
\end{lemma}

\begin{proof}
The action of $t$ on $A$  determines a matrix $T\in{GL(r,\mathbb{Z})}$
such that $T-I_r$ is also invertible.
Since the canonical map from $Aut(F(r))$ to $GL(r,\mathbb{Z})$ is onto,
there is an automorphism $\tau$ of $F(r)$ which abelianizes to $T$.
Let $G=F(r)\rtimes_\tau\mathbb{Z}$.
Then $G/G'\cong\mathbb{Z}$, $G'\cong{F(r)}$ and $G'/G''\cong{A}$.

The automorphisms $\tau$ can be realized by a based self-diffeomorphism 
$f_\tau$ of $N=\#^r(S^3\times{S^1})$, and so $G\cong\pi_1(M(f_\tau))$,
where $M(f_\tau)$ is the mapping torus of $f_\tau$.
Surgery on the section of $M(f_\tau)$ determined by the basepoint of $N$
gives a homology 4-sphere $\Sigma$,
and the cocore of the surgery is a fibred 2-knot in $\Sigma$ with group $G$.
\end{proof}

If we try a similar approach to constructing a knot in $S^4$, 
we loose some control of the commutator subgroup. 
However we can improve slightly on Theorem 1.

\begin{lemma} 
Let $A$ be a knot module which is free of rank $r$ as an abelian group.
Then $A$ is realized by a ribbon $2$-knot whose group
is an ascending HNN extension with base $F(r)$.
\end{lemma}

\begin{proof}
Since $A$ is a knot module, $t-1$ acts invertibly, 
and so there is a matrix $M\in{GL(r,\mathbb{Z})}$
such that $(t-1)X=MX$ for all $X\in{A}$.
Let $\mu$ be an automorphism of $F(r)$ which induces $M$ on $F(r)/F(r)'=\mathbb{Z}^r$.
Then the group $\pi$ with presentation
\[
\langle
{t,~x_1,\dots,x_r}\mid{tx_it^{-1}=x_i\mu(x_i)}~\forall1\leq{i}\leq{r}
\rangle
\]
is an ascending HNN extension with base $F(r)$,
and $\pi'/\pi''\cong{A}$.
As in Theorem 1, it has a deficiency-1 Wirtinger presentation,
and so is the group of a ribbon 2-knot.
\end{proof}

We may find such an automorphism $\mu$ by 
using row reduction to express $M$ as a product of elementary
matrices, and then lifting each factor in the obvious way.

The endomorphism $\theta$ determined by $\theta(x_i)=x_i\mu(x_i)$ 
induces an isomorphism on abelianiztion.
Since $H_2(F(r);\mathbb{Z})=0$, it also induces automorphisms of all 
the quotients $F(r)/F(r)_{[n]}$ of the lower central series \cite{St65}.
Can we choose $\mu$ so that $\theta$ is also an automorphism?

\section{some questions}

It is possible that every 2-knot with group of deficiency 1 may be a ribbon 2-knot.
If so, this would relate the somewhat opaque notion of deficiency 
to an immediately understandable geometric condition.
The following questions reflect our lack of a fuller understanding of the connections 
between these notions.

Three of these questions are closely related to the Andrews-Curtis Conjecture,
through \cite{Yo82}.
The algebraic version of this conjecture is that every balanced presentation 
$\mathcal{P}=\langle{x_1,\dots,x_m}\mid{w_1,\dots,w_m}\rangle$
of the trivial group may be reduced to the empty presentation by a limited class of moves
(AC {\it equivalence}).
We may replace relators by their inverses, conjugates and products with another relator,
and we may introduce or remove a generator $y$ and a relator $yz$ with $z$ a word in the
other generators.
In \cite{Yo82} it is shown that if $\mathcal{P}_o$ is a deficiency-1 presentation of a knot group $G$ and 
$\mu$ is a normal generator for $G$ then $G$ has a deficiency-1 Wirtinger presentation
in which $\mu$ is one of the generators if and only if $\mathcal{P}=\mathcal{P}_o\cup\{\mu=1\}$
is AC equivalent to the empty presentation.

\begin{enumerate}

\item{}If  $A$ is a knot module which is free of rank $r$ as an abelian group,
is it realized by a 2-knot with group $\pi$ such that $\pi'$ is free?

\item{}Do deficiency-1 knot groups have deficiency-1 Wirtinger presentations?

\item{}Is every knot group $\pi$ with $\pi'$ free the group of a ribbon 2-knot?

\item{}If $t\in\pi$ normally generates a knot group $\pi$ with a deficiency-1 Wirtinger presentation,
does $\pi$ have such a presentation for which $t$ is the image of a generator?

\item{}If $K$ is a ribbon 2-knot with group $\pi$ such that $\pi'$ is free is $K$ fibred?

\item{}Do deficiency-1 knot  groups have geometric dimension $\leq2$?

\item{}If $\pi$ is a high dimensional knot group such that $c.d.\pi=2$ is $g.d.\pi=2$?

\end{enumerate}

Question (1) may be reformulated as follows.
Let $P\in{GL(r,\mathbb{Z})}$ be a matrix such that $\det(P-I_r)=\pm1$.
Is there an automorphism $\tau$ of $F(r)$ with abelianization $P$
and such that $\langle{x_1,\dots,x_r}\mid{x_i=\tau(x_i)~\forall1\leq{i}\leq{r}}\rangle$
is a presentation of the trivial group?
The results of \S4 above are partial answers to this question.

Questions (2-4) would all have positive answers if the Andrews-Curtis Conjecture is true.
However our questions consider an a priori narrower class of presentations,
and so some may have positive answers even if 
the Andrews-Curtis Conjecture fails.
For instance,  (3) may be reformulated as follows.
Suppose that $\theta$ is an automorphism of $F(r)$ such that 
$\langle{x_1,\dots,x_r}\mid{x_i=\theta(x_i)~\forall1\leq{i}\leq{r}}\rangle$
is a presentation of the trivial group.
Is this presentation AC equivalent to the empty presentation?

If $\pi$ has a  Wirtinger presentation $\mathcal{P}$ of deficiency 1 and
$g$ is a generator then $\mathcal{P}\cup\{g=1\}$ is AC equivalent to the empty presentation.
Question (4) may be phrased as asking ``is this also so for $\mathcal{P}\cup\{t=1\}$"?
This question reflects the fact that knot groups may have infinitely many essentially distinct
normal generators \cite{Su85}.
Even if (2) holds it may not be true that 2-knots with deficiency-1 groups are ribbon knots.

Resolution of  (5) may depend on progress in 4-dimensional topology,
although it could also be viewed as a question about the geometric information encoded in a LOT.
(See \cite{Ho85}.)

Question (6) is part of the Whitehead Conjecture, 
while (7) is more of a wish for simplicity than an expectation.

\end{document}